\documentclass[english]{amsart}
\usepackage[T1]{fontenc}
\usepackage[latin9]{inputenc}
\usepackage{amstext}
\usepackage{amsthm}
\usepackage{amssymb}

\makeatletter
\numberwithin{equation}{section}
\numberwithin{figure}{section}
\theoremstyle{plain}
\newtheorem{thm}{\protect\theoremname}
\theoremstyle{definition}
\newtheorem{defn}[thm]{\protect\definitionname}
\theoremstyle{plain}
\newtheorem{prop}[thm]{\protect\propositionname}
\theoremstyle{plain}
\newtheorem{fact}[thm]{\protect\factname}
\theoremstyle{plain}
\newtheorem{cor}[thm]{\protect\corollaryname}
\theoremstyle{definition}
\newtheorem{problem}[thm]{\protect\problemname}
\theoremstyle{remark}
\newtheorem{rem}[thm]{\protect\remarkname}
\theoremstyle{plain}
\newtheorem*{thm*}{\protect\theoremname}
\theoremstyle{plain}
\newtheorem{lem}[thm]{\protect\lemmaname}

\AtBeginDocument{
  
}

\makeatother

\usepackage{babel}
\providecommand{\corollaryname}{Corollary}
\providecommand{\definitionname}{Definition}
\providecommand{\factname}{Fact}
\providecommand{\lemmaname}{Lemma}
\providecommand{\problemname}{Problem}
\providecommand{\propositionname}{Proposition}
\providecommand{\remarkname}{Remark}
\providecommand{\theoremname}{Theorem}

\begin{document}
\title[Computability of Finite Quotients]{Computability of Finite Quotients of Finitely Generated Groups}
\begin{abstract}
We study systematically groups whose marked finite quotients form
a recursive set. We give several definitions, and prove basic properties
of this class of groups, and in particular emphasize the link between
the growth of the depth function and solvability of the word problem.
We give examples of infinitely presented groups whose finite quotients
can be effectively enumerated. Finally, our main result is that a
residually finite group can be even not recursively presented and
still have computable finite quotients, and that, on the other hand,
it can have solvable word problem while still not having computable
finite quotients. 
\end{abstract}

\author{Emmanuel Rauzy}
\curraddr{Université de Paris. UFR de Mathématiques. Bâtiment Sophie Germain.
8 place Aurélie Nemours, 75013 Paris, France}
\email{emmanuel.rauzy.14@normalesup.org}
\maketitle

\section*{Introduction }

The fact that several well known conjectures which have been solved
for countable groups remain open for finitely presented groups, such
as the Burnside problem, or the existence of groups of intermediate
growth, shows that little is known about the specificities of finitely
presented groups. 

One of the most striking results that affects specifically finitely
presented groups is McKinsey's theorem : finitely presented residually
finite groups must have solvable word problem. (McKinsey's theorem,
which appeared in \cite{McKinsey1943}, is in fact set in a more general
setting than that of finitely presented residually finite groups,
we are only interested in its group theoretical version, which was
first made explicit by Mal'cev in \cite{Malcev1958}, followed by
Dyson in \cite{Dyson64} and by Mostoswski in \cite{Mostowski1966}.)
The proof of this theorem hinges on the fact that the finite quotients
of a finitely presented group can be enumerated. However, it is known
that recursively presented residually finite groups can have unsolvable
word problem, and thus quotients that cannot be enumerated: two examples
of these exist in the literature, one by Meskin (\cite{Meskin1974}),
which is in addition center-by-metabelian, and one by Dyson in \cite{Dyson1974}.
This proves in particular that there can be no Higman theorem for
general residually finite groups: not all recursively presented residually
finite groups embed in finitely presented residually finite groups. 

Before stating a precise definition of ``having computable finite
quotients'', let us recall the proof of McKinsey's theorem. 

Consider a finitely presented residually finite group $G$, with a
generating family $S$ of cardinal $n$, and $w$, a word whose letters
are elements of $S\cup S^{-1}$. We try to determine whether $w=e$
in $G$. 

First, as in any recursively presented group, we can apply to $w$
an algorithm that will stop if it is the identity element of the group,
and that never stops otherwise. This is done by enumerating relations,
and their conjugates, and the products of their conjugates, and checking
every time whether the word $w$ has appeared. 

Secondly, notice that if $F$ is a finite group, one can determine
in a finite number of steps whether $F$ is a quotient of $G$: this
is done by checking, for every generating family of $F$ of cardinal
$n$, whether the (finitely many) defining relations of $G$ hold
between those generators. Thus, from an enumeration of all finite
groups by their Cayley table, one can obtain an enumeration of all
finite quotients of $G$. In each of those quotients, we can check
whether or not the image of $w$ is trivial. If a quotient is found,
in which the image of $w$ is different from the identity, we can
conclude that in $G$ as well $w$ must be different from the identity,
and stop that procedure. By definition of a residually finite group,
any non-trivial element of $G$ will have a non-trivial image in a
finite quotient, thus that second part of the algorithm will always
stop if $w$ is not the identity element in $G$. 

This proof is the sum of three facts: 
\begin{enumerate}
\item In a recursively presented group, there is an algorithm that determines
when a word corresponds to the identity (and never stops otherwise). 
\item In a finitely presented group, there is an algorithm that determines
when a finite group is a quotient of it, and produces a morphism. 
\item In a residually finite group whose finite quotients can be enumerated,
there is an algorithm that determines when a word corresponds to a
non-identity element (and never stops otherwise). 
\end{enumerate}
The first point is an equivalence and is well known. The last point
is very natural, and the definition of ``residually finite group''
could have been introduced to answer the question: ``what is a sufficient
condition for a group whose finite quotients can be enumerated to
have solvable word problem?''. 

The second point is the one at the origin of this article, which consists
in a systematic study of groups whose finite quotients can be detected.
Although several articles (\cite{Bou-Rabee2016,Garrido2016,HARTUNG2011})
have already mentioned the fact that it is sometimes possible to enumerate
the finite quotients of non-finitely presented groups, the following
was never stated as a definition: 
\begin{defn}
A finitely generated group $G$, together with a generating set $S$,
is said to have Computable Finite Quotients (CFQ) if there is an algorithm
that, given a pair $(F,f)$, where $F$ is a finite group (given by
a finite presentation) and $f$ is a function from $S$ to $F$, determines
whether the function $f$ extends to a group morphism, that is whether
there exists a group homomorphism $\hat{f}:G\rightarrow F$ such that
for any $s$ in $S$, $f(s)=\hat{f}(s)$. 

If there exists an algorithm that terminates when the function $f$
extends to a group morphism, but does not terminate otherwise, we
say that $G$ has Recursively-enumerable Finite Quotients (ReFQ).

If there exists an algorithm that terminates when the function $f$
does not extend to a group morphism, but does not terminate otherwise,
we say that $G$ has co-Recursively-enumerable Finite Quotients (co-ReFQ). 
\end{defn}

Of course having CFQ is equivalent to having both ReFQ and co-ReFQ. 

It follows from the proof given earlier of McKinsey's algorithm that
all finitely presented groups have CFQ. This fact should be compared
to a result of Bridson and Wilton from \cite{Bridson2015}: there
is no algorithm that, given a finite presentation, decides whether
or not the group it defines admits a non-trivial finite quotient. 

This article builds upon Dyson's groups from \cite{Dyson1974} to
obtain the following two theorems:
\begin{thm}
\label{thm:Thm1}There exists a finitely generated residually finite
group with solvable word problem, but that has uncomputable finite
quotients. 
\begin{thm}
\label{thm:THM2}There exists a finitely generated residually finite
group with computable finite quotients, that still has unsolvable
word problem. 
\end{thm}

\end{thm}

We proceed as follows. 

In the first section, we give several equivalent definitions of CFQ
groups, and we explain how the profinite topology on a group can be
used to describe the various decisions problems that can be solved
using McKinsey-type algorithms. The second section quickly enumerates
some easy properties: free or direct products of groups with computable
finite quotients also have this property, etc. In the third section
interactions with the depth function for residually finite groups
are explained. The fourth section provides examples of infinitely
presented groups with CFQ: wreath products of groups with CFQ and
L-presented groups (including some well known torsion groups). In
the final section, we prove Theorem \ref{thm:Thm1}, thanks to Dyson's
method which uses doubles of the lamplighter group, and give two proofs
of Theorem \ref{thm:THM2}, one using the same technique, and another
one that relies on Slobodskoi's work in \cite{Slobodskoi1981}.

Following \cite{Dyson1974}, throughout this article, recursively
presented groups will be called re groups (for recursively enumerable),
and groups in which there is an algorithm that recognizes non-trivial
elements will be called co-re groups. A group has solvable word problem
if and only if it is re and co-re. 

\section{CFQ Groups}

\subsection{Equivalent definitions }

Just as a group with solvable word problem is a group in which words
in the generators corresponding to the identity can be enumerated
by an algorithm which respects a computable ordering on the set of
words in the generators, or a re group is a group in which these words
can be enumerated, but without any guarantee on the order of the enumeration,
groups with CFQ, ReFQ or co-ReFQ can be equivalently characterized
by enumeration of their finite quotients. Let us precise this. 

$G$ is still a group generated by $S$, of cardinal $n$. $S$ can
be seen as $\left\{ 1,\,...,n\right\} $. 

Call a $n$-marked finite group a pair $(F,f)$, where $F$ is a finite
group given by its Cayley table, and $f$ is a function from $\left\{ 1,...,n\right\} $
to $F$ whose image generates all of $F$. Consider an effective enumeration
$(F_{1},f_{1})$, $(F_{2},f_{2})$, $(F_{3},f_{3})$,... of all $n$-marked
finite groups, which satisfies $card(F_{n})\leq card(F_{n+1})$. (This
can be obtained by listing in order all possible Cayley tables, then
listing all $n$-tuples from those tables and determining when a tuple
defines a generating set). Define $\mathcal{A}_{G}\subseteq\mathbb{N}$
to be the set of indices $k$ for which $f_{k}$ defines a morphism
from $G$ to $F_{k}$. Then $G$ has CFQ, ReFQ or co-ReFQ if and only
if $\mathcal{A}_{G}$ is, respectively, a recursive set, a recursively
enumerable set or a co-recursively enumerable set. 

Of course, those definitions are independent of a choice of a generating
family. 
\begin{prop}
Having one of CFQ, ReFQ or co-ReFQ is independent of a choice of a
generating family. 
\end{prop}

\begin{proof}
Let $S$ and $T$ be two finite generating sets of a group $G$ (not
necessarily of the same cardinal). Fix for each $s$ in $S$ an expression
$s=t_{1}^{\alpha_{1}}...t_{k}^{\alpha_{k}}$, with $\alpha_{i}\in\left\{ -1,1\right\} $
and $t_{i}\in T$, that gives $s$ as a product of elements of $T$
or their inverses, and for each $t$ in $T$ an expression $t=s_{1}^{\beta_{1}}s_{2}^{\beta_{2}}...s_{k}^{\beta_{k}}$
that describes $t$ in terms of the generators of $S$ and their inverses.
For a finite group $F$ and a function $f$ from $S$ to $F$, define
the function $f'$ from $T$ to $F$ by $f'(t)=f(s_{1})^{\beta_{1}}...f(s_{k})^{\beta_{k}}$.
The function $f$ defines a homomorphism if, and only if, $f'$ also
defines a homomorphism $\varphi'$, that satisfies $\varphi'(s)=f(s)$
for $s$ in $S$. That last condition is an equality in $F$ that
can be tested using the expressions $s=t_{1}^{\alpha_{1}}...t_{k}^{\alpha_{k}}$,
even before it is known whether or not $f'$ extends. Using this,
all three properties, CFQ, ReFQ, co-ReFQ can be seen to be independent
of the chosen generating family of $G$.
\end{proof}

\subsection{\label{subsec:Variations-on-McKinsey's}Variations on McKinsey's
algorithm and the profinite topology }

It was already explained, when discussing McKinsey's original algorithm,
that the notion of residually finite group becomes a very natural
one to introduce when one asks for a sufficient condition for a group
with CFQ to be co-re. We can more generally search for conditions
on a finitely generated group $G$ with CFQ, that might allow to solve
various algorithmic problems. We will see that such conditions can
be expressed through the use of the profinite topology on $G$. In
what follows, we will say that we use McKinsey's algorithm to mean
that we enumerate all finite quotients of a group, checking some condition
in each quotient, and stopping when a finite quotient is found that
satisfies the required condition. 

Fix a finitely generated group $G$ with CFQ (or simply ReFQ). 

We will search for conditions that allow, given two disjoint subsets
$A$ and $B$ of $G$, to\emph{ distinguish $A$ from $B$} using
McKinsey's algorithm, that is to say to decide, given an element of
$G$ that belongs to $A\cup B$, whether it belongs to $A$ or to
$B$. 

The first condition we need is that, given a morphism from $G$ onto
a finite group $F$, it be possible to completely determine the images
of $A$ and of $B$ in $F$. We will thus need the following definition:
\begin{defn}
\label{def:DETERMINABLE in finite qutients}A subset $A$ of $G$
is said to be \emph{determinable in finite quotients of $G$} if there
exists an algorithm that, given a morphism $\phi$ from $G$ onto
a finite group $F$, can determine the image $\phi(A)$ (i.e. solve
the membership problem for $\phi(A)$ in $F$). 

A family $(A_{i})_{i\in\mathbb{N}}$ of subsets of $G$ is \emph{uniformly}
\emph{determinable in finite quotients of $G$ }if each $A_{i}$ is
determinable in finite quotients of $G$, and if the algorithm that
determines $A_{i}$ in finite quotients of $G$ depends recursively
of $i$. 
\end{defn}

Note, as an example, that a finitely generated subgroup $H$ of a
group $G$ is determinable in the finite quotients of $G$, because
the image of $H$ in a quotient of $G$ is the group generated by
the images of the generators of $H$. 

The property of being determinable in finite quotients is interesting
in itself, however we will not give it much attention in this paper.
We still remark the following.

It is easy to build a family $(A_{i})_{i\in\mathbb{N}}$ of subsets
of $\mathbb{Z}$, such that: it is uniformly recursively enumerable,
but not uniformly determinable in finite quotients of $\mathbb{Z}$.
But a stronger result will naturally appear in this paper as a byproduct
of the proof of Theorem \ref{thm:Thm1}:
\begin{prop}
There exists a recursive subset of $\mathbb{Z}$ that is not determinable
in finite quotients of $\mathbb{Z}$. 
\end{prop}

\begin{proof}
See Lemma \ref{lemma for thm 1}, disregarding the statement about
the subset of $\mathbb{Z}$ being closed. 
\end{proof}
On the other hand, it is easy to build a subset of $\mathbb{Z}$ that
is determinable in its finite quotients, but not recursive. 
\begin{prop}
\label{prop:Determinable not re}There exists a subset of $\mathbb{Z}$
that is not re, but still is determinable in finite quotients of $\mathbb{Z}$. 
\end{prop}

\begin{proof}
For a function $h$ that grows faster than any recursive function,
consider the enumeration $2h(1)$, $2h(2)+1$, $3h(3)$, $3h(4)+1$,
$3h(5)+2$, $4h(6)$,... This defines a set that is not re, but whose
image in any quotient $\mathbb{Z}/n\mathbb{Z}$ of $\mathbb{Z}$ is
all of $\mathbb{Z}/n\mathbb{Z}$. 
\end{proof}
Given two disjoint subsets $A$ and $B$ of $G$ that are indeed determinable
in finite quotients of $G$, the profinite topology of $G$ can be
used to decide whether McKinsey's algorithm can tell them apart. 

The \emph{profinite topology} on a group $G$ was introduced in \cite{Hall1950},
it is the topology whose open basis consists of cosets of finite index
normal subgroups of $G$. We denote $\mathcal{PT}(G)$ the profinite
topology on $G$. A closed subset of $G$ in $\mathcal{PT}(G)$ is
called a \emph{separable} set. We note $\bar{A}$ the closure of a
set $A$ in $\mathcal{PT}(G)$. The following easy facts render explicit
the link between the profinite topology and McKinsey-type algorithms:
\begin{fact}
A subset $A$ of a group $G$ is open in $\mathcal{PT}(G)$ if and
only if for any element $a$ of $A$, there is a morphism $\phi$
from $G$ onto a finite group $F$ such that $\phi^{-1}(\phi(a))\subseteq A$,
i.e. such that if an element of $G$ has the same image as $a$ in
$F$, it also belongs to $A$. 
\begin{fact}
\label{fact:Closure in PT(G)}The closure $\bar{B}$ of a subset $B$
of $G$ is the biggest set of elements that satisfy the following
condition:
\end{fact}

\begin{center}
For any morphism $\phi$ from $G$ onto a finite group $F$, $\phi(\bar{B})\subseteq\phi(B)$. 
\par\end{center}

\end{fact}

This shows that the closure of a set $B$ in $\mathcal{PT}(G)$ is
precisely the set of elements that cannot be distinguished from $B$
using McKinsey's algorithm. 

We can now give conditions that allow two disjoint subsets $A$ and
$B$ of $G$ to be distinguishable by McKinsey's algorithm. 
\begin{prop}
\label{prop:Profinite Result}Let $A$ and $B$ be disjoint subsets
of $G$ that are determinable in finite quotients of $G$. Then McKinsey's
algorithm can be used to distinguish $A$ from $B$ if and only if
the following two conditions hold: 
\[
A\cap\bar{B}=\varnothing
\]
\end{prop}

\[
\bar{A}\cap B=\varnothing
\]

\begin{proof}
This is a simple consequence of Fact \ref{fact:Closure in PT(G)}. 
\end{proof}
McKinsey's original result on residually finite groups can thus be
interpreted as an application of this proposition to singletons, and
a residually finite group is precisely a group in which all the singletons
are closed in the profinite topology. 

Two other well studied families of groups fall in the range of Proposition
\ref{prop:Profinite Result}: conjugacy separable groups, and LERF
groups. 

A conjugacy separable group is a group $G$ in which all the conjugacy
classes are separable. It is easy to see that the conjugacy classes
of a finitely generated group $G$ are always uniformly determinable
in finite quotients of $G$ (when a class $C$ is given by any of
its elements), thus Proposition \ref{prop:Profinite Result} can be
used to distinguish conjugacy classes.

A LERF group (for locally extended residually finite), or subgroup
separable group, is a group $G$ whose finitely generated subgroups
are separable. 

We then have the following proposition:
\begin{prop}
\label{Prop MCKINSEY}Let $G$ be a finitely generated group with
ReFQ. 
\begin{itemize}
\item If $G$ is  residually finite, then it is co-re. If it is re and residually
finite, it has solvable word problem. 
\item If $G$ is conjugacy separable, then there exits an algorithm that
decides when two of its elements are not conjugate. If it is re and
conjugacy separable, then it has solvable conjugacy problem.
\item If $G$ is LERF, there is an algorithm that, given a tuple $(x_{1},...,x_{n},g)$
of elements of $G$, stops exactly when $g$ does not belong to the
subgroup of $G$ generated by $(x_{1},...,x_{n})$. If it is re and
LERF, it has solvable generalized word problem. 
\end{itemize}
\end{prop}

\begin{proof}
All three points follow from Proposition \ref{prop:Profinite Result},
noticing that it provides a uniform algorithm, and using it respectively
with:
\begin{itemize}
\item $A$ and $B$ being singletons,
\item $A$ and $B$ being conjugacy classes,
\item $A$ being a finitely generated subgroup of $G$ and $B$ a singleton. 
\end{itemize}
In each case, one needs to use the fact that those sets are uniformly
determinable in finite quotients, which is straightforward. 
\end{proof}
Because re groups naturally have co-ReFQ (Proposition \ref{prop:re=00003D>coREFQ}),
the statements that concern re group in the previous proposition could
be formulated with CFQ instead of ReFQ, without loss of generality. 

Proposition \ref{Prop MCKINSEY} follows in a very straightforward
way from the definitions of residually finite, of conjugacy separable
and of LERF groups, and it is surprising that the study of these properties
was not followed by a systematic study of the properties CFQ and ReFQ.
The author could point the lector to papers where it is implied that
re conjugacy separable groups always have solvable conjugacy problem,
which led him to believe it has to be ascertained that not all recursively
presented groups, or residually finite groups, or even residually
finite groups with solvable word problem, have ReFQ. 

\subsection{Membership problem for finite index normal subgroups}

In the article \cite{Bou-Rabee2016}, Bou-Rabee and Seward use, in
the course of a proof (\cite{Bou-Rabee2016}, Proof of Theorem 2),
the fact that, if a group $G$ has solvable ``membership problem
for finite index normal subgroups'', (or ``generalized word problem
for finite index normal subgroups'') and solvable word problem, then
it admits an algorithm that recognizes its finite quotients. 

We will now show that for re groups, having solvable membership problem
for finite index normal subgroups is actually equivalent to having
CFQ. This will allow us to give another point of view on groups with
CFQ, and at the same time making explicit the link to the isomorphism
problem for finite groups given by recursive presentations, which
we will sum up in the next sub-section. 

When formulating the membership problem for finite index normal subgroups,
it is implicit that the normal subgroup is given by a finite generating
family. 

Indeed, the membership problem for finite index normal subgroups in
a group $G$ asks for an algorithm that, given a tuple $(x_{1},...,x_{k},g)$
of elements of $G$, the first $k$ elements of which generate a finite
index normal subgroup of $G$, will decide whether $g$ belongs to
that subgroup. As opposed to that, when working with property CFQ,
we describe non-ambiguously a finite index normal subgroup $N$ of
$G$ by a pair $(F,f)$, where $F$ is a finite group and $f$ a function
from the generators of $G$ to $F$, which extends to a group homomorphism,
the kernel of which is precisely $N$. 

Of course, given that second description, the problem ``does $g$
belong to $N$'' is solved by computing the image of $g$ in $F$
to see whether it is the identity of $F$. Thus a group in which one
can go from the description of a normal subgroup by generators to
a description of this subgroup by a morphism necessarily has solvable
membership problem for finite index normal subgroups. We will see
that for re groups this is also sufficient. 

On the other hand, given a description by morphism of the subgroup
$N$, that is a morphism $\varphi:G\rightarrow F$ with $ker\,\varphi=N$,
one can always obtain a description of it by generators, as one can
effectively carry out the well known proof of Schreier's lemma, which
is often used to prove that a finite index subgroup of a finitely
generated group is itself finitely generated. Indeed, if $S$ is a
generating family of $G$, for any $x$ in $F$ and $s$ in $S$,
a preimage $\hat{x}$ of $x$ can be found in $G$, by exhaustive
search, and a preimage of $x\varphi(s)$ can be found as well, call
it $\hat{y}$. Schreier's lemma asserts that the elements of the form
$\hat{x}s\hat{y}^{-1}$ generate $N$. 

This allows us to prove the following (the backward implication is
directly adapted from \cite{Bou-Rabee2016}):
\begin{prop}
\label{prop:Property-ReFQ-is}Property ReFQ is equivalent to having
co-re membership problem for finite index normal subgroups, that is
to having an algorithm that decides when an element is not in a given
finite index normal subgroup, and does not terminate otherwise. 
\end{prop}

\begin{proof}
Suppose first that $G$ has ReFQ, and let $N$ be a finite index normal
subgroup of $G$ generated by a family $x_{1},...,x_{k}$. Let finally
$g$ be an element of $G$, we want to decide whether $g$ belongs
to $N$. Enumerate the quotients $(F,f)$ of $G$, and look for a
finite quotient in which the image of $g$ is non-trivial, while the
images of $x_{1},...,x_{n}$ are all trivial. If $g$ does not belong
to $N$, such a quotient exists (the projection $G\rightarrow G/N)$,
and this algorithm will terminate. 

Now suppose $G$ has co-re membership problem for finite index normal
subgroups. Write $\langle S\vert R\rangle$ a presentation of $G$.
Let $(F,f)$ be a finite group together with a function from $S$
to $F$. As $f$ does not necessarily define a morphism, we cannot
yet apply Schreier's method. But if $\mathcal{F}_{n}$ is a free group
with basis the $n$ generators of $G$, $f$ does define a morphism
$\varphi$ from $\mathcal{F}_{n}$ to $F$, and thus we can find a
family $x_{1},...,x_{k}$ of elements of $\mathcal{F}_{n}$ that generate
$ker(\varphi)$. $F$ is given by the presentation: $\langle S\vert x_{1},...,x_{k}\rangle$.
(But $x_{1}$, ..., $x_{k}$ generate $ker\,\varphi$ as a group,
and not only as a normal subgroup as would be guaranteed by any presentation
of $F$ on the generators $S$). 

Now $f$ extends to a morphism if and only if $F$ satisfies all the
relations of $G$, that is to say if and only if the relations $x_{1},...,x_{k}$
imply the relations that appear in $R$, that is to say if and only
if $\langle S\vert R,\,x_{1},...,x_{k}\rangle$ is just another presentation
of $F$. But this is a presentation of $G/N$, where $N$ is the subgroup
of $G$ generated by $x_{1},...,x_{k}$. If $f$ does not extend to
a morphism, $G/N$ is a strict quotient of $F$. 

Thus we can do the following: enumerate the elements of $G$, $g_{1},\,g_{2},...$
Then use the membership algorithm for $N$, (which, as we suppose,
can only show something does not belong to $N$), to find elements
that define different classes in $G/N$, that is: find $g_{i_{0}}$
that does not belong to $N$, then $g_{i_{1}}$ which is such that
neither itself nor $g_{i_{0}}g_{i_{1}}^{-1}$ belong to $N$, and
$g_{i_{2}}$ such that $g_{i_{2}}$, $g_{i_{0}}g_{i_{2}}^{-1}$ and
$g_{i_{1}}g_{i_{2}}^{-1}$ don't belong to $N$... If $F$ is a quotient
of $G$, this method will yield $card(F)$ elements, at which point
the algorithm has proven that $F$ is a quotient of $G$. Of course,
if $F$ is not a quotient of $G$, it will never stop. 
\end{proof}
In a re group, determining whether $g$ belongs to the subgroup generated
by $x_{1},...,x_{k}$ can always be done when $g$ belongs to that
group, thus having co-re membership problem for finite index normal
subgroups is equivalent to having solvable membership problem for
finite index normal subgroups. Similarly, re groups always have co-ReFQ
(Proposition \ref{prop:re=00003D>coREFQ}), thus for such a group
ReFQ and CFQ are equivalent. This yields:
\begin{cor}
For recursively presented groups, having CFQ and having solvable membership
problem for finite index normal subgroups are equivalent properties. 
\end{cor}

We can use this to show that in a re group with CFQ, from the description
of a finite index normal subgroup $N$ by a generating family $x_{1},...,x_{k}$,
one can deduce a pair $(F,f)$, where $F$ is a finite group and $f$
extends to a morphism $\varphi$ of $G$ to $F$ with kernel $N$. 

Launch two procedures, one is the same as that described in the proof
above: enumerate elements of $G$ that define different cosets of
$G/N$. We get successively better lower bounds on $card(G/N)$: $card(G/N)\geq1,2,3,...$

The other procedure gives upper bounds on the size of $G/N$. Start
from the enumeration of all marked finite groups $(F_{1},f_{1})$,
$(F_{2},f_{2})$,... For each pair $(F_{i},f_{i})$, test whether
$G/N$ is a quotient of the group $F_{i}$ according to one of the
finitely many left inverses of $f_{i}$. This can be done because
$G/N$ is given by a recursive presentation (as we add finitely many
relations to a presentation of $G$ which we suppose re), thus there
is an algorithm that tests whether the finitely many relations of
a finite group $F$ are satisfied in $G/N$, and terminates when indeed
they are. This procedure yields upper bounds on the cardinal of $G/N$. 

At some point, the lower and upper bounds will agree, and we will
know that the pair $(F,f)$ that has $card(F)=card(G/N)$ defines
an isomorphism $F\simeq G/N$, and thus the normal subgroup $N$ is
described by the pair $(F,f)$. 

\subsection{Isomorphism problem for finite groups}

Note that another condition for CFQ appears clearly in the course
of the proof of Proposition \ref{prop:Property-ReFQ-is}: at some
point, it is known that the presentation $\langle S\vert R,\,x_{1}=e,...,x_{k}=e\rangle$
is the presentation of a finite group (even, that it is a quotient
of the given group $F$), and the question ``is $F$ a quotient of
$G$'' is equivalent to ``is this finite group a strict quotient
of $F$''. It follows from this remark: 
\begin{prop}
\label{prop:isomorphismfinitegroups}A group $G$, which admits a
presentation $\langle S\vert R\rangle$, has CFQ if the isomorphism
problem is solvable for the following family of presentations: all
finite presentations of finite groups, and all presentations of the
form $\langle S\vert R,\,R_{1}\rangle$, where $R_{1}$ is a finite
set of relations such that $\langle S\vert R_{1}\rangle$ is finite. 
\end{prop}

The isomorphism problem for finite groups is solvable, this is well
known, but it only means that we can determine when two finite groups
given by finite presentations are isomorphic, and the question here
is to determine whether these two groups, one given by a finite presentation,
and the other one by an infinite presentation, are isomorphic. It
can be seen that the finite presentations of finite groups can actually
be omitted in Proposition \ref{prop:isomorphismfinitegroups}. If
$G$ is a re group which does not have CFQ (and such groups exist,
see Theorem \ref{thm:Thm1}), this shows in particular: 
\begin{cor}
There exists a recursive family of recursive presentations of finite
groups, for which the isomorphism problem is unsolvable. Moreover
any two of those presentations differ only by a finite number of relations. 
\end{cor}

Since it was remarked that it suffices to be able to obtain lower
bounds of the cardinal of the groups given by these presentations
to solve their isomorphism problem, the world problem is not uniformly
solvable for this family of presentation. 

\section{Basic properties }

We will now quickly establish some basic results about the three properties
CFQ, ReFQ and co-ReFQ. 

\subsection{Recursively presented groups}

It is easy to see that re groups have co-ReFQ. 
\begin{prop}
\label{prop:re=00003D>coREFQ}Any recursively presented group has
co-ReFQ. 
\end{prop}

\begin{proof}
Let $G$ be a re group generated by $S$ of cardinal $n$. Let $(F,f)$
be a $n$-marked finite group. For any relation $r$ of $G$, write
$r=s_{1}^{\text{\ensuremath{\alpha_{1}}}}...s_{k}^{\text{\ensuremath{\alpha_{k}}}}$
with $\alpha_{i}\in\{-1;1\}$ and $s_{i}\in S$, we again test the
equality $e=f(s_{1})^{\alpha_{1}}...f(s_{k})^{\alpha_{k}}$. Since
we suppose $G$ re, this can be carried out successively on all relations
of $G$. If $f$ does not extend to a morphism, a relation true in
$G$ but not in $F$ will eventually be found. 
\end{proof}
This is remarkable because it is more often the case that algorithmic
problems for groups be naturally re for re groups, than co-re (conjugacy
problem, generalized word problem, isomorphism problem for finitely
presented groups, etc). 

\subsection{Hereditarity }
\begin{prop}
If $G$ and $H$ both have one of CFQ, ReFQ, co-ReFQ, then so does
their free product $G*H$. 
\end{prop}

\begin{proof}
Note that we have shown that those properties are independent of the
generating family, thus we can show this using as a generating family
of $G*H$ the union of a generating family of $G$ and of one of $H$.
The proof is then straightforward, as a function from that generating
family to a finite group extends to a morphism of $G*H$, if and only
if both restrictions to $G$ and to $H$ extend as morphisms. 
\end{proof}
\begin{prop}
\label{prop:ReFQ finite index sub}ReFQ is inherited by finite index
subgroups.
\end{prop}

\begin{proof}
Let $G$ be a group, and $N$ a finite index subgroup of $G$. An
enumeration of the finite quotients of $G$ gives an enumeration of
finite quotients of $N$, by restricting the homomorphisms to $N$.
Not all quotients of $N$ need arise this way, but add the following:
whenever a pair $(F,f)$ is found that defines a quotient of $N$,
list all quotients of the finite group $F$, and when a quotient $F\overset{\pi}{\rightarrow}F_{0}$
is found, add $(F_{0},\pi\circ f)$ to the list of quotients of $N$.
We claim that all finite quotients of $N$ arise this way. 

Let $M$ be a finite index normal subgroup of $N$. Then $M$ is of
finite index in $G$, it may not be normal in $G$, but it contains
a normal subgroup $M_{0}$ which is both finite index and normal in
$G$. Then $G\rightarrow G/M_{0}$ restricts to a morphism $N\rightarrow N/M_{0}$,
and $M/M_{0}$ is a normal subgroup in $N/M_{0}$, and the quotient
of $N/M_{0}$ by $M/M_{0}$ is, of course, $N/M$. 
\end{proof}
Note that any countable group embeds in a two generated simple group,
and that simple groups always have CFQ. This of course shows that
CFQ is not inherited by subgroups. Note that the author doesn't know
of a residually finite group with CFQ, with a subgroup without CFQ. 
\begin{problem}
Find a finitely presented residually finite group with a finitely
generated subgroup that does not have CFQ. 
\end{problem}

This problem is relevant in the search of a corrected Higman Embedding
Theorem for residually finite groups, i.e. in the search of necessary
and sufficient conditions for a finitely generated residually finite
group to be a subgroup of a finitely presented residually finite group. 

The following proposition provides insight into what happens when
computing the quotient of a group with CFQ. 
\begin{prop}
\label{prop add relations}Let $G$ be a finitely generated group.
Let $H$ be a group obtained by adding finitely many relations and
identities to $G$. If $G$ has any of ReFQ, co-ReFQ or CFQ, then
so does $H$. 
\end{prop}

\begin{proof}
Let $G$ and $H$ be two finitely generated groups, with a morphism
$\pi$ from $G$ onto $H$. Let $(F,f)$ be a marked finite group.
It is obvious that if $f$ extends to a homomorphism $\text{\ensuremath{\varphi}}$
from $H$ onto $F$, then $f\circ\pi$ extends as well to a morphism
$\psi$, which, in addition, satisfies $ker(\pi)\subseteq ker(\psi)$.
On the other hand, if $f\circ\pi$ extends to a morphism $\psi$,
and if $ker(\pi)$ is contained in $ker(\psi)$, then $\psi$ factors
through $\pi$ and $f$ will extend to a morphism. The diagram is
the following: 
\[
\begin{array}{ccc}
G & \overset{\pi}{\rightarrow} & H\\
 & \overset{\,\,\,\,\,_{\psi}}{\searrow} & \downarrow^{_{\varphi}}\\
 &  & F
\end{array}
\]
Thus if $G$ has CFQ or ReFQ, we can reduce the finite quotient question
of $H$ to a question about subgroups inclusion. 

If $ker(\pi)$ is finitely generated as a normal subgroup, then the
question ``is $ker(\pi)$ contained in $ker(\psi)$?'' can be solved
in finite time, as it is solved by computing $\psi(r)$ for each $r$
in a generating family of $ker(\pi)$. If $ker(\pi)$ is generated
by an identity -that is a set of relations of the form $v(x_{1},...,x_{k})$,
where $v$ is a element of the free group on $k$ generators, and
$x_{1},...,x_{k}$ take all possible values of $G^{k}$- this question
can also be answered, because to check whether an identity holds in
a finite group, one only needs to check finitely many relations. 

If $G$ is a co-ReFQ group, we can determine whether $ker(\pi)$ is
contained in $ker(\psi)$ even without knowing if $\psi$ defines
a morphism from $G$, and thus if $(F,f)$ does not define a quotient
of $H$, we will either prove that $f\circ\pi$ does not extend to
a quotient of $G$, or that, even if it were to define a quotient,
the inclusion of kernels would not hold. With this, all cases of Proposition
\ref{prop add relations} are covered. 
\end{proof}
Since free groups obviously have CFQ, this proves again that finitely
presented groups have CFQ, and the improvement due to Mostowski \cite{Mostowski1966}
which asserts that groups defined by finitely many relations and identities
have CFQ. This implies for instance that all finitely generated metabelian
groups, while not being necessarily finitely presented, have CFQ,
because a result of Hall (\cite{Hall1954}) implies that any finitely
generated metabelian group can be presented with the metabelian identity
($\forall x\forall y\forall z\forall t\left[\left[x,y\right]\left[z,t\right]\right]=e$)
together with finitely many relations. 

The following is a direct consequence of Proposition \ref{prop add relations}: 
\begin{cor}
The fundamental group of a graph of groups with vertex groups that
have CFQ (or ReFQ or co-ReFQ) and finitely generated edge groups also
has CFQ (respectively ReFQ or co-ReFQ). This includes free products
amalgamated over finitely generated subgroups and HNN extensions over
finitely generated subgroups. 

A direct product of groups that have one of CFQ, ReFQ or co-ReFQ,
again has that property.
\end{cor}

\subsection{Groups with the same finite quotients }

Although having CFQ is defined for any finitely generated group, this
property is very much attached to residually finite groups, and not
only because of the interaction between having CFQ and having solvable
word problem. For a group $G$, define its \emph{finitary image }(the
name comes from \cite{Dyson1974}) to be the quotient of $G$ by its
finite residual, i.e. by the intersection of all its finite index
subgroups, or again, the image of $G$ in its profinite completion
$\hat{G}$. Note this group $G_{f}$. It is the biggest residually
finite quotient of $G$. Note $\pi$ the morphism $G\rightarrow G_{f}$.
Fix a finite marked group $(F,f)$. This is the same situation as
described when investigating quotients of CFQ groups, except that
we have the property, which follows from the universal property of
$G_{f}$: any morphism $\psi$ of $G$ to a finite group $F$ factors
through $\pi$, that is for any morphism $\psi$ to a finite group,
$ker(\pi)\subseteq ker(\psi)$. The situation is summed up in the
diagram: 
\[
\begin{array}{ccc}
G & \overset{\pi}{\rightarrow} & G_{f}\\
 & \overset{\,\,\,\,\,_{\psi}}{\searrow} & \downarrow^{_{\varphi}}\\
 &  & F
\end{array}
\]

It follows that $G$ has any of CFQ, ReFQ, co-ReFQ if and only if
$G_{f}$ shares the same property. From this, the following proposition
is obvious. 
\begin{prop}
\label{prop: same quotients}If two finitely generated groups $G$
and $H$ admit the same finitary image, then $G$ has any of CFQ,
ReFQ, co-ReFQ if and only if $H$ shares the same property. 
\end{prop}

We say that a group has \emph{trivially computable finite quotients}
if its finitary image is a finite group. For instance, simple groups,
or finite exponent groups (because of the solution to the restricted
Burnside problem), all have trivially CFQ. Note again that since a
finitely generated simple group can be neither re nor co-re, CFQ groups
can have as bad algorithmic properties as desired. Obtaining groups
with CFQ that still have bad algorithmic properties becomes challenging
only amongst residually finite groups. This is what was obtained in
Theorem \ref{thm:THM2}. 

In \cite{Dyson1974} are constructed two groups with solvable word
problem, the finitary image of one is co-re but not re, while the
finitary image of the second is re but not co-re. This last group
cannot have CFQ, otherwise its finitary image would have CFQ while
being residually finite, thus it would be co-re. Thus this is a first
example of a group with solvable word problem but without CFQ. Note
however that this construction relies on the fact that the constructed
group is not residually finite in order to prove that it does not
have CFQ; while Theorem \ref{thm:Thm1} provides a group with solvable
word problem and without CFQ that is itself residually finite. 

\section{Relation with the Depth Function for residually finite groups\label{sec:Relation-with-the}}

In \cite{Bou-Rabee2010}, Bou-Rabee introduced the\emph{ residual
finiteness growth function}, or \emph{depth function}, $\rho_{G}$,
of a residually finite group $G$. We briefly recall its definition. 

Fix a generating family $S$ of $G$. Consider the function $\rho_{S}$,
that to a natural number $n$ associates the smallest number $k$
such that, for any non-trivial element of length at most $n$ in $G$,
there exists a finite quotient of $G$ of order at most $k$, such
that the image of this element in that quotient is non-trivial. This
is the depth function of $G$ with respect to the generating family
$S$. 

For two functions $f$ and $g$ defined on natural numbers, note $f\succeq g$
to mean that there exists a constant $C$ such that for any number
$n$, $f(n)\geq Cg(Cn)$. It is easy to see that one defines an equivalence
relation $\simeq$ by putting $f\simeq g$ if and only if $f\succeq g$
and $g\succeq f$. 

If $S$ and $S'$ are two different generating families of a group
$G$, the functions $\rho_{S}$ and $\rho_{S'}$ will satisfy $\rho_{S}\simeq\rho_{S'}$
(\cite{Bou-Rabee2010}). Thus one can define uniquely the depth function
$\rho_{G}$ of the group $G$ by considering the equivalence class
for $\simeq$ of a function $\rho_{S}$, for some generating family
$S$. 

The interaction between the depth function of a group and it having
CFQ makes it worth mentioning here, and in fact, an ancestor of the
depth function can be found in McKinsey's original article, \cite{McKinsey1943},
where an upper bound for what would be a ``depth function'' for
lattices (partially ordered sets) is computed. This interaction also
appears in \cite{Bou-Rabee2016}. 

Consider a residually finite group $G$, that has CFQ, generated by
a family $S$. We know that $G$ is then co-re, because McKinsey's
algorithm applies: list all quotients of $G$ in order, and check
wether an element has non-trivial image in one of those quotients.
How long this will take is bounded by the depth function. In particular,
if for an element $w$ of length $n$ of $G$, the algorithm has already
tested all quotients of size at most $\rho_{S}(n)$, and not found
a quotient in which $w$ is non-trivial, then $w=e$. 
\begin{prop}
\label{Prop depth dehn}Let $G$ be a finitely generated residually
finite group with CFQ.

If there exists a generating family $S$ of $G$ and a recursive function
$h$ that satisfies $\rho_{S}\leq h$, then $G$ has solvable word
problem. 

If $G$ has solvable word problem, then for any generating family
$S$ of $G$, the depth function of $G$ with respect to $S$ is recursive. 
\end{prop}

\begin{proof}
The first claim follows from the sentence that immediately precedes
this proposition. The second claim is straightforward. 
\end{proof}
\begin{rem}
The depth function $\rho_{G}$ of $G$ is defined up to the equivalence
relation $\simeq$. In the statement of this last proposition appears
the condition ``there exists a recursive function $h$ that satisfies
$\rho_{S}\leq h$''. If two functions $f$ and $g$ satisfy $f\simeq g$,
and if there exists a recursive function $h$ that satisfies $f\leq h$,
then of course there exists a recursive function $h'$ that satisfies
$g\le h'$, and thus it makes sense to say that there exists a recursive
function that bounds the depth function $\rho_{G}$. Note however
that it is possible for a recursive function $f$ to be equivalent,
for $\simeq$, to a non-recursive function $g$, (and even worse:
the equivalence class for $\simeq$ of any recursive function contains
a non-recursive function), and because of this, the term ``recursive
depth function'', must be manipulated carefully. In what follows,
we will say that a group $G$ has \emph{recursive depth function}
if for any generating family $S$ of the group $G$, the function
$\rho_{S}$ is recursive. Proposition \ref{Prop depth dehn} implies
that for a group with CFQ, either all its relative depth functions
$\rho_{S}$ are recursive, or none of them are. From this it is natural
to ask the following problem: 
\begin{problem}
Let $G$ be a finitely generated residually finite group, and $S$
and $S'$ two generating families of $G$. Is it possible that exactly
one of $\rho_{S}$ and $\rho_{S'}$ be recursive?
\end{problem}

\end{rem}

The following follows immediately from Proposition \ref{Prop depth dehn}:
\begin{cor}
For residually finite groups with CFQ, having solvable word problem
is equivalent to having recursive depth function. 
\end{cor}

Theorem \ref{thm:THM2} thus asserts the existence of a group with
CFQ but non-recursive depth function. Note that result similar to
this corollary exists, that relates the Dehn function of a finitely
presented group and solvability of the word problem in it. 

In \cite{Kharlampovich2017}, it was shown that for any recursive
function $f$, there is a finitely presented residually finite group
with depth function greater than $f$, and yet with word problem solvable
in polynomial time. For such a group, McKinsey's algorithm is far
from being optimal. The group constructed in order to prove Theorem
\ref{thm:Thm1} shows that, for non-finitely presented groups, the
situation can be even more extreme: in it, the word problem is solvable,
but not by McKinsey's algorithm.

It was explained in the first section of this paper that conjugacy
separability and subgroup separability, when confronted to the property
of having computable finite quotients, play a role similar to residual
finiteness, but with respect to the conjugacy problem and the generalized
word problem. It is then natural to introduce functions similar to
the depth function of residually finite groups, but that quantify
conjugacy separability and subgroup separability. This was first done
in \cite{Lawton2017} for conjugacy separability, and in \cite{Dere2018}
for subgroup separability. What was said in this section about recursiveness
of the depth function translates easily to those two functions. 

\section{Examples of groups with CFQ }

In this section, we give examples of infinitely presented groups that
have CFQ. Those examples contain some commonly used infinitely presented
groups. 

\subsection{Wreath products }

The (restricted) wreath product of two groups $G$ and $H$, noted
$H\wr G$, is the semi direct product $H\ltimes\underset{H}{\bigoplus}G$,
where $H$ acts on $\underset{H}{\bigoplus}G$ by permuting the indices.
The wreath product $H\wr G$ of two finitely generated groups is always
finitely generated, however, by a theorem of Baumslag (\cite{Baumslag1961}),
it is finitely presented only if $G$ is finitely presented and $H$
is finite (excluding the case where $G$ is the trivial group). If
$\langle S_{G}\vert R_{G}\rangle$ and $\langle S_{H}\vert R_{H}\rangle$
are presentations respectively of $G$ and $H$, with $S_{G}\cap S_{H}=\varnothing$,
then a presentation of $H\wr G$ is given by the following: 
\[
\langle S_{G},S_{H}\vert R_{G},R_{H},\left[hg_{1}h^{-1},g_{2}\right],(g_{1},g_{2},h)\in S_{G}\times S_{G}\times(H\setminus\left\{ 1\right\} )\rangle
\]
(where by $h\in H\setminus\left\{ 1\right\} $ we actually mean ``for
each $h\in H\setminus\left\{ 1\right\} $, choosing a way of expressing
$h$ in terms of a product of the generators of $S_{H}$''). 

The following proposition gives necessary and sufficient conditions
for a wreath product to have CFQ. 
\begin{prop}
The wreath product $H\wr G$ of two finitely generated groups $G$
and $H$ has CFQ if and only if either $H$ has CFQ and is infinite,
or if $H$ is finite and $G$ has CFQ. 
\end{prop}

\begin{proof}
Suppose that $H\wr G$ has CFQ. From the presentation of $H\wr G$
given above, it is easy to see that $H$ can be obtained as a quotient
of $H\wr G$ by adding only finitely many relations to it. Thus by
Proposition \ref{prop add relations}, $H$ has CFQ. 

Suppose additionally that $H$ is finite. A finite group $F$, marked
by a function $f$, is a quotient of $G$ if and only if the function
$g$, defined as the identity on the generators of $H$ and as $f$
on those of $G$, which thus sends a generating family of $H\wr G$
to one of $H\wr F$, can be extended to a morphism. But because $H$
is finite, $H\wr F$ is also finite, and thus it can be determined
whether or not it is a quotient of $H\wr G$. This proves that $G$
has CFQ. 

For the converse, suppose first that $H$ is finite and that $G$
has CFQ. Because $H$ is finite, a presentation for $H\wr G$ can
be obtained by adding finitely many relations to a presentation of
$G*H$, which has CFQ because both $G$ and $H$ have. Thus Proposition
\ref{prop add relations} applies again. 

Suppose now that $H$ is infinite and has CFQ. In a finite quotient
of $H\wr G$, some non-trivial element of $H$ must necessarily have
a trivial image. But it is easy to see, from the presentation of $H\wr G$
given above, that, for any element $h\ne1$ of $H$, the relation
$h=1$ together with the relations of $H\wr G$ will always also imply
that $\left[g_{1},g_{2}\right]=1$, for any pair $(g_{1},g_{2})$
of elements of $G$. This implies that any finite quotient of $H\wr G$
is in fact a finite quotient of $H\wr G_{ab}$, where $G_{ab}$ denotes
the abelianization $G/\left[G,G\right]$ of $G$. It thus suffices
to prove that if $H$ has CFQ and $G$ is abelian, then the wreath
product $H\wr G$ also has CFQ. 

Consider a finite group $F$ together with a function $f$ that goes
from a generating family of $H\wr G$ to $F$. To decide whether $f$
can be extended to a morphism, one can first check whether the restriction
$f_{H}$ of $f$ to a generating family of $H$ extends as a morphism
from $H$ onto a subgroup $F_{H}$ of $F$, using the fact that $H$
has CFQ. If it does not, then $f$ cannot be extended to a group morphism.
If it does, one can compute a presentation of $F_{H}$. $F$ is then
a quotient of $H\wr G$ if and only if it is a quotient of $F_{H}\wr G$.
But $G$ is finitely presented, being abelian, and $F_{H}$ is finite,
thus this wreath product is finitely presented, and a finite presentation
of it can be obtained from the presentation of $F_{H}$. This presentation
can in turn be used to decide whether $F$ is a quotient of $F_{H}\wr G$,
and thus of $H\wr G$. 
\end{proof}

\subsection{L-presented groups }

L-presented groups were introduced in \cite{Bartholdi2003}. An L-presentation,
or finite endomorphic presentation, is a quadruple $\langle S\vert R_{1}\vert\varPhi\vert R_{2}\rangle$,
where $S$ is a finite set of generating symbols, $R_{1}$ and $R_{2}$
are finite sets of relations, i.e. of elements of the free group $\mathcal{F}_{S}$
defined over $S$, and $\varPhi$ is a finite set of endomorphisms
of $\mathcal{F}_{S}$. Such a presentation defines a presentation,
in the usual sense of the term, by adding to $R_{1}\cup R_{2}$, as
further relations, all elements of $\mathcal{F}_{S}$ that can be
obtained from $R_{2}$ by iterating the endomorphisms of $\varPhi$.
(This process can be carried out effectively, an L-presentation is
thus always a recursively enumerable presentation.) 

It was proven in \cite{HARTUNG2011} that a coset enumeration process
can be carried out in L-presented groups, and that the membership
problem for finite index subgroups is solvable in L-presented groups.
This, we have seen, directly implies that those groups have CFQ. 
\begin{prop}
L-presented groups have CFQ. 
\end{prop}

Many groups of interest that are not finitely presented were proven
to admit L-presentations, including groups of intermediate growth,
and various groups acting on rooted trees, like the Gupta--Sidki
group (see \cite{Bartholdi2003}). 

Note that many of those groups were proven to be conjugacy separable
(\cite{Wilson1997}) or subgroup separable (\cite{GRIGORCHUK2003,Garrido2016}).
The fact that they have CFQ thus completes for those groups the proof
of solvability of the conjugacy problem, or of the generalized word
problem. (Although in some cases, solutions to these problems are
known, that do not rely on McKinsey's algorithm). 

\section{Main unsolvability results}

In this section, we prove the two main theorems of this paper.

Theorem \ref{thm:THM2} follows from the construction that Slobodskoi
used in \cite{Slobodskoi1981} to prove undecidability of the universal
theory of finite groups. However, the exact result we use was implicit
in \cite{Slobodskoi1981}, and it was first pointed out to hold by
Bridson and Wilton in \cite{Bridson2015}. We first show how, thanks
to this result, the proof of Theorem \ref{thm:THM2} is immediate. 

We then detail Dyson's construction from \cite{Dyson1974}, which,
with \cite{Meskin1974}, contained the first examples of re residually
finite groups without solvable word problem. We show that this construction
provides us with another proof of Theorem \ref{thm:THM2}, which has
the advantage of relying on groups whose structure is very explicit.
We finally use Dyson's construction to build a residually finite group
with solvable word problem, and yet without CFQ, thus proving Theorem
\ref{thm:Thm1}. 

\subsection{Slobodskoi's example for Theorem \ref{thm:THM2}}

Recall that the finitary image $G_{f}$ of a group $G$ is its biggest
residually finite quotient. 

To each finite presentation, one can associate a residually finite
group, by taking the finitary image of the group defined by this presentation.
If $\pi$ is a finite presentation for a group $G$, and if $G_{f}$
is the finitary image of $G$, say that $\pi$ is a finite presentation
of $G_{f}$ \emph{as a residually finite group, }and that $G_{f}$
is\emph{ finitely presentable as a residually finite group.} 

In 1981, in \cite{Slobodskoi1981}, Slobodskoi proved that the universal
theory of finite groups is undecidable. As a consequence of his paper,
groups that are finitely presented\emph{ }as residually finite groups\emph{
}do not have uniformly solvable word problem. But in fact, more can
be deduced from the proof that appears in \cite{Slobodskoi1981},
as was made explicit by Bridson and Wilton in \cite{Bridson2015}: 
\begin{thm}
[\cite{Bridson2015}, Theorem 2.1]\label{thm: Slobodskoi}There
exists a finitely presented group $G$ in which there is no algorithm
to decide which elements have trivial image in every finite quotient. 
\end{thm}

Two other equivalent formulations of this theorem are the following
ones:
\begin{thm*}
There exists a finitely presented group $G$ whose finitary image
is not recursively presented. 
\begin{thm*}
There exists a group $H$ which is finitely presentable as a residually
finite group, but which has unsolvable word problem. 
\end{thm*}
\end{thm*}
This theorem then directly implies Theorem \ref{thm:THM2}. 
\begin{proof}
[Proof of Theorem \ref{thm:THM2}] A group which is finitely presented
as a residually finite group must have CFQ, by Proposition \ref{prop: same quotients}.
Thus if such a group has unsolvable word problem, it satisfies the
criteria required by Theorem \ref{thm:THM2}. A group, finitely presented
as a residually finite group, with unsolvable word problem exists
by Theorem \ref{thm: Slobodskoi}. 
\end{proof}

\subsection{Dyson's Groups }

These groups are amalgamated products of two lamplighter groups. 

The lamplighter group $L$ is the wreath product $\mathbb{Z}\wr\mathbb{Z}/2\mathbb{Z}$
of $\mathbb{Z}$ and of the order two group. It admits the following
presentation:
\[
\langle a,\varepsilon\vert\,\varepsilon^{2},\,\left[\varepsilon,a^{i}\varepsilon a^{-i}\right],i\in\mathbb{Z}\rangle
\]
The element $a^{i}\varepsilon a^{-i}$ of $L$ corresponds to the
element of $\underset{\mathbb{Z}}{\bigoplus}\mathbb{Z}/2\mathbb{Z}$
with only one non-zero coordinate in position $i\in\mathbb{Z}$. We
call it $u_{i}$. Consider another copy $\hat{L}$ of the lamplighter
group, together with an isomorphism from $L$ to $\hat{L}$ we note
$g\mapsto\hat{g}$. For each subset $\mathcal{A}$ of $\mathbb{Z}$,
define $L(\mathcal{A})$ to be the amalgamated product of $L$ and
$\hat{L}$, with $u_{i}=a^{i}\varepsilon a^{-i}$ identified with
$\hat{u}_{i}=\hat{a}^{i}\hat{\varepsilon}\hat{a}^{-i}$ for each $i$
in $\mathcal{A}$. It has the following presentation:
\[
\langle a,\hat{a},\varepsilon,\hat{\varepsilon}\vert\,\varepsilon^{2},\,\hat{\varepsilon}^{2},\,\left[\varepsilon,a^{i}\varepsilon a^{-i}\right],\left[\hat{\varepsilon},\hat{a}^{i}\hat{\varepsilon}\hat{a}^{-i}\right],i\in\mathbb{Z},\,a^{j}\varepsilon a^{-j}=\hat{a}{}^{j}\hat{\varepsilon}\hat{a}{}^{-j},\,j\in\mathcal{A}\rangle
\]

Recall that the \emph{double} of a group $G$ over a subgroup $H$
is the free product of two copies of $G$ amalgamated along the identity
over $H$, denote it $G*_{H}G$. Dyson's group associated to $\mathcal{A}$
is thus the double of the Lamplighter group $L$ over the subgroup
generated by elements of the form $a^{i}\varepsilon a^{-i}$, for
$i$ in $\mathcal{A}$. Call $H_{\mathcal{A}}$ this subgroup. 

For $n$ a non-zero natural number, denote $L(\mathcal{A})_{n}$ the
group $\langle L(\mathcal{A})\,\vert\,a^{n},\hat{a}^{n}\rangle$. 

Call $\mathcal{A}\,mod\,n$ the set $\left\{ r\in\left\{ 0,...,n-1\right\} ,\exists a\in\mathcal{A},a\equiv r\,mod\,n\right\} $. 
\begin{lem}
\label{lem:-L(A)n-fpquotient}$L(\mathcal{A})_{n}$ \textup{is finitely
presented and residually finite. It is the amalgamated product of
two copies of the finite} \emph{wreath product $\mathbb{Z}/n\mathbb{Z}\wr\mathbb{Z}/2\mathbb{Z}$,
and it admits the following presentation:}
\end{lem}

\begin{align*}
\langle a,\hat{a},\varepsilon,\hat{\varepsilon}\vert\,a^{n},\,\hat{a}^{n},\,\varepsilon^{2},\,\hat{\varepsilon}^{2},\,\left[\varepsilon,a^{i}\varepsilon a^{-i}\right],\left[\hat{\varepsilon},\hat{a}^{i}\hat{\varepsilon}\hat{a}^{-i}\right],0\leq i\leq n-1,\,\,\,\,\,\,\,\,\,\,\\
\,a^{j}\varepsilon a^{-j}=\hat{a}{}^{j}\hat{\varepsilon}\hat{a}{}^{-j},\,j\in\mathcal{A}\,mod\,n\rangle
\end{align*}

\begin{proof}
The given presentation is obtained from the presentation of $L(\mathcal{A})$,
adding relations $a^{n}$ and $\hat{a}{}^{n}$, and simplifying the
relations as can be done. It then follows from that presentation that
$L(\mathcal{A})_{n}$ is an amalgamated product of two copies of the
finite wreath product $\mathbb{Z}/n\mathbb{Z}\wr\mathbb{Z}/2\mathbb{Z}$.
Finally it is well known that an amalgamated product of finite groups
is residually finite (see \cite{Magnus1969}). 
\end{proof}
The properties of the group $L(\mathcal{A})$ are stated using the
profinite topology on $\mathbb{Z}$. A basis of open sets for $\mathcal{PT}(\mathbb{Z})$
consists in sets of the form $n+p\mathbb{Z}$, for $n$ and $p$ integers.
Thus a subset $\mathcal{A}$ of $\mathbb{Z}$ is open in $\mathcal{PT}(\mathbb{Z})$
if and only if for every $n$ in $\mathcal{A}$ there exists an integer
$p$ such that $n+p\mathbb{Z}\subseteq\mathcal{A}$. 

We can now state the properties of the group $L(\mathcal{A})$ that
are relevant to this work. 
\begin{prop}
\label{prop:Main-prop-1-2-3}Let $\mathcal{A}$ be a subset of $\mathbb{Z}$. 
\begin{enumerate}
\item $L(\mathcal{A})$ is re, co-re or has solvable word problem if and
only if $\mathcal{A}$ is respectively re, co-re or recursive.
\item $L(\mathcal{A})$ is residually finite if and only if $\mathcal{A}$
is closed in $\mathcal{PT}(\mathbb{Z})$. 
\item $L(\mathcal{A})$ has CFQ if and only if the function which to $n$
associates $\mathcal{A}\,mod\,n$ is recursive. 
\end{enumerate}
\end{prop}

The first two points of this proposition were proven in \cite{Dyson1974}.
We still include proofs for those two statements. The proof given
for (1) is exactly that of Dyson. We then give two proofs of (2),
one shorter and more conceptual than that of \cite{Dyson1974}, which
was suggested to us by the anonymous referee, and another one which,
although not as clear as the previous one, can readily be rendered
effective, and in the proof of Proposition \ref{prop:Last Prop},
we refer ourselves to this second proof. 

Notice finally that the condition which appears in (3) could be stated
as: $\mathcal{A}$ is determinable in finite quotients of $\mathbb{Z}$
(see Definition \ref{def:DETERMINABLE in finite qutients}). 
\begin{proof}
We prove all three points in order. 

If $L(\mathcal{A})$ is re or co-re, then clearly so is $\mathcal{A}$,
as $n$ belongs to $\mathcal{A}$ if and only if $u_{n}=\hat{u}{}_{n}$
in $L(\mathcal{A})$, which proves one direction of (1).

It is clear that if $\mathcal{A}$ is re, the presentation of $L(\mathcal{A})$
given above is re as well. 

Suppose now that $\mathcal{A}$ is co-re. We can enumerate the complement
of $\mathcal{A}$, and thus enumerate elements of the form:
\[
(*)\,\,\,\,w=a^{\alpha_{1}}x_{1}\hat{a}{}^{\beta_{1}}y_{1}...a^{\alpha_{k}}x_{k}\hat{a}{}^{\beta_{k}}y_{k}z
\]
where $\alpha_{i},\,\beta_{i}\in\mathbb{Z}$, $x_{1}$, ..., $x_{k}$
and $y_{1}$, ..., $y_{k}$ are elements of the base groups of $L$
and $\hat{L}$ that have null components corresponding to indices
in $\mathcal{A}$, and $z$ is any element in $L$ or in $\hat{L}$.
The elements written this way are exactly the elements in normal form
for the amalgamated product $L(\mathcal{A})$. Recall that the normal
form in an amalgamated product (\cite{Lyndon2001}) necessitates two
choices of transversals, here one for $L/(\underset{\mathbb{\mathcal{A}}}{\bigoplus}\mathbb{Z}/2\mathbb{Z})$
and one for $\hat{L}/(\underset{\mathbb{\mathcal{A}}}{\bigoplus}\mathbb{Z}/2\mathbb{Z})$,
and that an element in normal form is a consecutive product of an
element of one transversal, followed by one of the other, etc, terminated
by any element of one of these groups. But elements $a^{\alpha}x$,
with $\alpha$ in $\mathbb{Z}$ and $x$ in the base group with null
components on $\mathcal{A}$, indeed form the most natural transversal
for $L/(\underset{\mathbb{\mathcal{A}}}{\bigoplus}\mathbb{Z}/2\mathbb{Z})$.
Thus any non-trivial element $g$ of $L(\mathcal{A})$ is equal to
exactly one element in this enumeration. Ideally, we would then enumerate
words that give the identity in $L(\mathcal{A})$, and listing words
that can be obtained concatenating a word in normal form to a word
that defines the identity would give the desired enumeration. Since
$L(\mathcal{A})$ is not supposed re, we cannot directly enumerate
this set of trivial words, but we will over-approximate it by a re
set. For $w$ as in $(*)$, note $B_{w}$ the set of all indices that
appear in elements $x_{i}$ or $y_{i}$ of the base groups ($B_{w}$
can be empty). The over approximation consists of the words corresponding
to the identity in $L(\mathbb{Z}\setminus B_{w})$. Note that $w$
is also in normal form in $L(\mathbb{Z}\setminus B_{w})$, thus non-trivial
there. Of course, $B_{w}$ is finite, thus $\mathbb{Z}\setminus B_{w}$
is re, thus as we already remarked, we can enumerate words (in $a$,
$\hat{a}$, $\varepsilon$, $\hat{\varepsilon}$) that correspond
to the identity in $L(\mathbb{Z}\setminus B_{w})$. Then for any such
word $w_{1}$, the product $ww_{1}$ corresponds to a non-identity
element of $L(\mathbb{Z}\setminus B_{w})$, thus to a non-identity
element in $L(\mathcal{A})$, since, as $\mathcal{A\subseteq}(\mathbb{Z}\setminus B_{w})$,
$L(\mathbb{Z}\setminus B_{w})$ satisfies more relations than $L(\mathcal{A})$.
Thus enumerating products $ww_{1}$ with $w_{1}=e$ in $L(\mathbb{Z}\setminus B_{w})$
will only yield non-identity elements in $L(\mathcal{A})$. What's
more, every element of the form $ww_{2}$, where $w_{2}$ is a word
that is the identity in $L(\mathcal{A})$, will arise this way, again
because $L(\mathbb{Z}\setminus B_{w})$ satisfies more relations than
$L(\mathcal{A})$. 

Because the algorithm that enumerates relations in $L(\mathbb{Z}\setminus B_{w})$
depends recursively on $w$, this process can be applied simultaneously
to all words $w$ in normal form, giving an enumeration of all words
that correspond to non-identity elements of $L(\mathcal{A})$, thus
proving that $L(\mathcal{A})$ is co-re. 

Finally $\mathcal{A}$ is recursive if and only if it is both re and
co-re, if and only if $L(\mathcal{A})$ has solvable word problem.

This ends the proof of (1). 

We first sketch a conceptually simple proof of (2). 

Recall that the group $L(\mathcal{A})$ is the double of the lamplighter
group $L$ over the subgroup $H_{\mathcal{A}}$ generated by elements
of the form $a^{i}\varepsilon a^{-i}$, for $i$ in $\mathcal{A}$.
The statement (2) is implied by the two following arguments:
\begin{itemize}
\item the double of a residually finite group $G$ over a subgroup $H$
is itself residually finite if and only if $H$ is separable in $G$;
\item the subgroup $H_{\mathcal{A}}$ is separable in $L$ if and only if
$\mathcal{A}$ is separable in $\mathbb{Z}$. 
\end{itemize}
Those arguments are easy to check, and because a second proof for
(2) follows, we do not detail them. 

We now give another proof of (2), which can be easily rendered effective.
We in fact prove slightly more than (2): for a subset $\mathcal{A}$
of $\mathbb{Z}$, the finitary image $L(\mathcal{A})_{f}$ of $L(\mathcal{A})$
is $L(\overline{\mathcal{A}})$, where $\overline{\mathcal{A}}$ denotes
the closure of $\mathcal{A}$ in $\mathcal{PT}(\mathbb{Z})$. First
we show that if $n$ belongs to $\overline{\mathcal{A}}$, then in
any finite quotient $(F,f)$ of $L(\mathcal{A})$, the images of $u_{n}$
and $\hat{u}{}_{n}$ are the same in $F$. Let $(F,f)$ be some finite
quotient of $L(\mathcal{A})$. Call $p$ and $p'$ the orders of $f(a)$
and $f(\hat{a})$ in $F$. Then, because $n$ belongs to $\overline{\mathcal{A}}$,
$n+pp'\mathbb{Z}$ must meet with $\mathcal{A}$, as it is a neighborhood
of $n$. Thus we have $k$ such that $n+pp'k\in\mathcal{A}$, that
is, such that $u_{n+pp'k}=\hat{u}{}_{n+pp'k}$ in $L(\mathcal{A})$.
Then, in $F$ (we omit to write the homomorphism onto $F$): 
\begin{align*}
u_{n} & =(a^{p})^{p'k}u_{n}(a{}^{p})^{-p'k}=a^{pp'k}u_{n}a^{-pp'k}\\
 & =u_{n+pp'k}\\
 & =\hat{u}{}_{n+pp'k}\\
 & =(\hat{a}{}^{p'})^{pk}\hat{u}{}_{n}(\hat{a}{}^{p'})^{-pk}=\hat{u}{}_{n}
\end{align*}
This shows that $L(\mathcal{A})_{f}$ is a quotient of $L(\overline{\mathcal{A}})$.
It is then sufficient to see that $L(\overline{\mathcal{A}})$ is
residually finite to see that $L(\mathcal{A})_{f}=L(\overline{\mathcal{A}})$. 

We suppose that $\mathcal{A}$ is closed to omit the closure notation.
Let $w$ be a non-identity element of $L(\mathcal{A})$, and write
it in normal form $w=a^{\alpha_{1}}x_{1}\hat{a}{}^{\beta_{1}}y_{1}...a^{\alpha_{k}}x_{k}\hat{a}{}^{\beta_{k}}y_{k}z$,
as in the proof of (1). 

Suppose first that the normal form is the trivial one: $w=z$ with
$z$ in $L$ or in $\hat{L}$. Then $w$ is non-trivial in the quotient
of $L(\mathcal{A})$ obtained by identifying the two copies $L$ and
$\hat{L}$ of the lamplighter group (i.e. $\langle L(\mathcal{A})\,\vert\,a=\hat{a},\varepsilon=\hat{\varepsilon}\rangle$
), which is just the lamplighter group itself, which is residually
finite.

We can now suppose the normal form has several terms. Each $x_{i}$
is an element of $\underset{\mathbb{Z}\setminus\mathcal{A}}{\bigoplus}\mathbb{Z}/2\mathbb{Z}$,
that is to say a product $\prod u_{k_{i,j}}$ with $k_{i,j}\notin\mathcal{A}$.
Because $\mathcal{A}$ is closed, for each such $k_{i,j}$ there is
$p_{i,j}$ that satisfies $(k_{i,j}+p_{i,j}\mathbb{Z})\cap\mathcal{A}=\varnothing$.
Similarly, for each $\hat{x}_{i}$, introduce integers $p'_{i,j}$,
$j=1,2,...$. Call $N$ the product $\prod p_{i,j}\prod p'_{i,j}$.
(It is $1$ if, for all $i$, $x_{i}=\hat{x}{}_{i}=e$.) We claim
that $w$ is non-trivial in $L(\mathcal{A})_{N}$. Indeed, $N$ was
chosen so that for any $(i,j)$, $k_{i,j}$ (or its remainder modulo
$N$) does not belong to $\mathcal{A}\,mod\,N$. This implies that
$w$ is also in normal form in $L(\mathcal{A})_{N}$ (by Lemma \ref{lem:-L(A)n-fpquotient}),
and thus non-trivial there. 

Again by Lemma \ref{lem:-L(A)n-fpquotient}, $L(\mathcal{A})_{N}$
is residually finite, so we've proven that $L(\mathcal{A})$ is residually
residually finite, which of course is the same as residually finite. 

Finally we prove (3). 

Suppose $\mathcal{A}\,mod\,n$ depends recursively of $n$. Let $(F,f)$
be a finite group together with a function $f$ from $\left\{ a,\,\hat{a},\,\varepsilon,\thinspace\hat{\varepsilon}\right\} $
to $F$. To determine whether $f$ defines a homomorphism, compute
the orders of $f(a)$ and of $f(\hat{a})$, and let $n$ be their
product. If $f$ extends to a morphism, this morphism factors through
the projection $\pi:L(\mathcal{A})\rightarrow L(\mathcal{A})_{n}$.
By Lemma \ref{lem:-L(A)n-fpquotient}, a finite presentation for $L(\mathcal{A})_{n}$
can be found from the computation of $\mathcal{A}\,mod\,n$. It can
then be determined in finite time from this presentation whether $f$
defines a homomorphism from $L(\mathcal{A})_{n}$ to $F$. 

Suppose now that $L(\mathcal{A})$ has CFQ. Let $n$ be a natural
number. To compute $\mathcal{A}\,mod\,n$, consider all possible presentations
for $L(\mathcal{A})_{n}$: for $B\subset\left\{ 0,...,n-1\right\} $,
define the presentation $\prod_{B}$: 
\begin{align*}
\langle a,\hat{a},\varepsilon,\hat{\varepsilon}\vert\,a^{n},\,\hat{a}^{n},\,\varepsilon^{2},\,\hat{\varepsilon}^{2},\,\left[\varepsilon,a^{i}\varepsilon a^{-i}\right],\left[\hat{\varepsilon},\hat{a}^{i}\hat{\varepsilon}\hat{a}^{-i}\right],0\leq i\leq n-1,\,\,\,\,\,\,\,\,\,\,\\
\,a^{j}\varepsilon a^{-j}=\hat{a}{}^{j}\hat{\varepsilon}\hat{a}{}^{-j},\,j\in B\rangle
\end{align*}
All these presentations define residually finite groups, and because
they are finitely presented, they have CFQ. $L(\mathcal{A})_{n}$
also has CFQ, because it is obtained from $L(\mathcal{A})$ by adding
two relations, thus we can start to enumerate the quotients of $L(\mathcal{A})_{n}$.
Also start enumerating the quotients of all groups given by the presentations
$\prod_{B}$, for $B\subset\left\{ 0,...,n-1\right\} $. Those $2^{n}$
lists are all different (because, as the presentations $\prod_{B}$
give residually finite groups, a list contains a finite group in which
the images of $a^{j}\varepsilon a^{-j}$ and $\hat{a}{}^{j}\hat{\varepsilon}\hat{a}{}^{-j}$
differ if and only if $j$ does not belong to $B$), and only one
corresponds to the list of quotients of $L(\mathcal{A})_{n}$. It
can be determined, in a finite number of steps, which of those lists
corresponds to $L(\mathcal{A})_{n}$, and thus which presentation
$\prod_{B}$ gives a presentation of $L(\mathcal{A})_{n}$, and then
one can conclude that $B=\mathcal{A}\,mod\,n$. 
\end{proof}
From Proposition \ref{prop:Main-prop-1-2-3}, to prove Theorem \ref{thm:Thm1},
it suffices to build $\mathcal{A}$ with the following properties:
$\mathcal{A}$ is closed in $\mathcal{PT}(\mathbb{Z})$, $\mathcal{A}$
is recursive, there is no algorithm that takes $n$ as input and computes
$\mathcal{A}\,mod\,n$. Similarly, to prove Theorem \ref{thm:THM2},
it suffices to build $\mathcal{A}$ such that: $\mathcal{A}$ is closed
in $\mathcal{PT}(\mathbb{Z})$, $\mathcal{A}$ is not recursive, but
there is an algorithm that, given $n$ as input, computes $\mathcal{A}\,mod\,n$.

\subsection{Building subsets of $\mathbb{Z}$ with prescribed properties}

We first give an alternative proof of Theorem \ref{thm:THM2}, before
completing the proof of Theorem \ref{thm:Thm1}. 
\begin{lem}
\label{lemma for thm 2}There exists a non-recursive subset $\mathcal{A}$
of $\mathbb{Z}$, closed in $\mathcal{PT}(\mathbb{Z})$, for which
$\mathcal{A}\,mod\,n$ depends recursively of $n$. 
\end{lem}

Note that without the closeness assumption, this result would be a
lot easier: it is precisely the result of Proposition \ref{prop:Determinable not re},
whose proof is very short, and yields a set that can be neither re
nor co-re. However, for a closed set $\mathcal{A}$, the computation
of $\mathcal{A}\,mod\,n$ will yield an enumeration of the complement
of $\mathcal{A}$: indeed, if $a$ is not in $\mathcal{A}$, some
open set $a+b\mathbb{Z}$ must not meet $\mathcal{A}$, and thus $a$
is not in $\mathcal{A}\,mod\,b$. This proves that if $\mathcal{A}$
is closed, and if $\mathcal{A}\,mod\,n$ is computable, then $\mathcal{A}$
is co-re. This is just the translation for Dyson's groups of: if $G$
is residually finite, and has CFQ, then $G$ is co-re. 
\begin{proof}
We construct a set $\mathcal{B}$, which will be the complement of
the announced $\mathcal{A}$. Thus it has to be open, re but not co-re,
and for any $a$ and $b$, the question ``is $a+b\mathbb{Z}$ a subset
of $\mathcal{B}$'' has to be solvable in a finite number of steps.
Indeed, $a$ belongs to $\mathcal{A}\,mod\,b$ if and only if $a+b\mathbb{Z}$
meets $\mathcal{A}$, if and only if $a+b\mathbb{Z}$ is not a subset
of the complement of $\mathcal{A}$. 

Call $p_{n}$ the $n$-th prime number. Define two sequences $(x_{n})_{n\geq0}$
and $(y_{n})_{n\geq1}$ by the following: 
\begin{align*}
x_{0} & =1\\
x_{n} & =p_{n}x_{n-1}^{2}\\
y_{n} & =x_{n-1}
\end{align*}

These sequences have the following properties: 
\begin{itemize}
\item for any $n$, $x_{n}\vert x_{n+1}$ and $y_{n}\vert y_{n+1}$. 
\item for any integer $b$, there is some (computable) $n$ such that $b\vert x_{n}$
and $b\vert y_{n}$. 
\item $p_{k}$ divides $x_{n}$ if and only if $k\geq n$, and $p_{k}$
divides $y_{n}$ if and only if $k>n$. 
\item for integers $k$, $k'$, $n$, $n'$, with $k\leq n$ and $k'\leq n'$,
$y_{k}+x_{n}\mathbb{Z}$ and $y_{k'}+x_{n'}\mathbb{Z}$ are disjoint
if and only if $k\neq k'$, and otherwise one is a subset of the other. 
\end{itemize}
All these are clear, the fourth point follows from the third, by remarking
that elements of $y_{k}+x_{n}\mathbb{Z}$ are all multiples of $p_{0}$,
$p_{1}$,... $p_{k-1}$, but none of them is a multiple of $p_{k}$. 

Consider a recursive function $f$ whose image is re but not co-re.
Assume that for any $n$, $1\le f(n)\leq n$ (it is easy to see that
such a function exists). Then we define $\mathcal{B}$ as the union:
\[
\mathcal{B}=\underset{n\in\mathbb{N}^{*}}{\bigcup}y_{f(n)}+x_{n}\mathbb{Z}
\]
Since $f$ is a recursive function, $\mathcal{B}$ is re. It is not
co-re, however, because $y_{m}$ belongs to $\mathcal{B}$ if and
only if $m$ belongs to the image of $f$ (this follows directly from
the properties of the sequences $(x_{n})_{n\geq0}$ and $(y_{n})_{n\geq1}$).

$\mathcal{B}$ is an open set, because it is defined as an union of
open sets. 

All that is left to see is that we can decide, for $a$ and $b$ integers,
whether $a+b\mathbb{Z}$ is a subset of $\mathcal{B}$. Suppose that
$a<b$. If $a=0$, then $0\in a+b\mathbb{Z}$, but $0\notin\mathcal{B}$,
thus $a+b\mathbb{Z}$ is not a subset of $\mathcal{B}$. If $a$ is
non-zero, no element of $a+b\mathbb{Z}$ is divisible by $b$. Thus,
because of the second property of the sequences $(x_{n})_{n\geq0}$
and $(y_{n})_{n\geq1}$ quoted above, there exists $N$ such that
if $N\leq k\leq n$, then $a+b\mathbb{Z}\cap y_{k}+x_{n}\mathbb{Z}=\varnothing$.
Thus $a+b\mathbb{Z}$ is a subset of $\mathcal{B}$ if and only if
it is a subset of the set $\mathcal{B}_{N}$, defined by:
\[
\mathcal{B}_{N}=\underset{n\in\mathbb{N}^{*},\,f(n)\leq N}{\bigcup}y_{f(n)}+x_{n}\mathbb{Z}
\]
Define a pseudo-inverse $g$ of $f$ by $g(m)=\inf\left\{ n,\,f(n)=m\right\} $.
Because we chose $f$ such that for any $n$, $f(n)\leq n$, for any
$m$, $g(m)\geq m$. If $m$ is not in the image of $f$, put $g(m)=\infty$.
The set $\mathcal{B}_{N}$ can  be expressed as the disjoint union:
\[
\mathcal{B}_{N}=\underset{k\in Im(f),\,k\leq N}{\bigcup}y_{k}+x_{g(k)}\mathbb{Z}
\]
Because $x_{k}\vert x_{g(k)}$, $\mathcal{B}_{N}$ is contained in
the set $\mathcal{C}_{N}$, defined by: 
\[
\mathcal{C}_{N}=\underset{k\leq N}{\bigcup}y_{k}+x_{k}\mathbb{Z}
\]
It can be determined whether $a+b\mathbb{Z}$ is contained in $\mathcal{C}_{N}$,
because the sequences $(x_{n})_{n\geq1}$ and $(y_{n})_{n\geq1}$
can be computed. If $a+b\mathbb{Z}$ is not contained in $\mathcal{C}_{N}$,
then it is not contained in $\mathcal{B}_{N}$ either.

If it is contained in $\mathcal{C}_{N}$, $a+b\mathbb{Z}$ is contained
in $\mathcal{B}_{N}$ if and only if, for each $k$, $a+b\mathbb{Z}\cap y_{k}+x_{k}\mathbb{Z}$
is contained in $\mathcal{B}_{N}$. But, because $\mathcal{B}_{N}$
and $\mathcal{C}_{N}$ are disjoint unions, $a+b\mathbb{Z}\cap y_{k}+x_{k}\mathbb{Z}$
is contained in $\mathcal{B}_{N}$ if and only if it is contained
in $y_{k}+x_{g(k)}\mathbb{Z}$. (If $k$ is not in $Im(f)$, $g(k)=\infty$,
by convention $y_{k}+x_{g(k)}\mathbb{Z}=\left\{ y_{k}\right\} $.)
Now this question can be effectively answered. If $a+b\mathbb{Z}\cap y_{k}+x_{k}\mathbb{Z}$
is empty, there is nothing to do. Otherwise, it is of the form $t+\text{lcm}(b,\,x_{k})\mathbb{Z}$.
Enumerate $f(1)$, $f(2)$,..., $f(\text{lcm}(b,x_{k}))$. Either
$k$ is in that list, in which case $g(k)$ can be computed and the
question ``is $a+b\mathbb{Z}\cap y_{k}+x_{k}\mathbb{Z}$ contained
in $y_{k}+x_{g(k)}\mathbb{Z}$'' can be settled, or $k$ does not
appear in the enumeration, which shows that $g(k)$ is greater than
$\text{lcm}(b,\,x_{k})$. It this last case, as $x_{g(k)}$ is greater
than $g(k)$, $y_{k}+x_{g(k)}\mathbb{Z}$ cannot contain any set of
the form $t+\text{lcm}(b,\,x_{k})\mathbb{Z}$. 
\end{proof}
This Lemma allows us to give another proof of Theorem \ref{thm:THM2}: 
\begin{proof}
[Proof of Theorem \ref{thm:THM2}, alternate] It follows from Proposition
\ref{prop:Main-prop-1-2-3} that the group $L(\mathcal{A})$, where
$\mathcal{A}$ is the set constructed in Lemma \ref{lemma for thm 2},
satisfies the requirements of Theorem \ref{thm:THM2}. 
\end{proof}
This group has a depth function which cannot be smaller than a recursive
function. We now prove the last lemma which ends the proof of Theorem
\ref{thm:Thm1}:
\begin{lem}
\label{lemma for thm 1}There exists a recursive subset $\mathcal{A}$
of $\mathbb{Z}$, closed in $\mathcal{PT}(\mathbb{Z})$, for which
$\mathcal{A}\,mod\,n$ does not depend recursively of $n$. 
\end{lem}

\begin{proof}
Call $p_{n}$ the $n$-th prime number. Fix some effective enumeration
$M_{1}$, $M_{2}$,... of all Turing machines. Consider the following
process: start running simultaneously all those machines, as is done
to show that the halting problem is re. While running calculations
on the $n$-th machine, at each new step in the computation, produce
a new power of $p_{2n}$: $p_{2n}$, $p_{2n}^{2}$, $p_{2n}^{3}$...
If the computation on this machine stops after $k$ steps, end the
list $p_{2n}$, $p_{2n}^{2}$, ..., $p_{2n}^{k}$ already produced
with $p_{2n+1}^{k+1}$. 

Call $\mathcal{A}$ the set of all powers of prime numbers obtained
this way. $\mathcal{A}$ is obviously re, as it was defined by an
effective enumeration process. It is even recursive. Indeed, for a
number $x$, if $x$ is not the power of a prime, then $x$ is not
in $\mathcal{A}$. If it is the power of a prime of even index, say
$x=p_{2n}^{k}$, then $x$ belongs to $\mathcal{A}$ if and only if
the $n$-th Turing machine does not stop in less than $k$ calculations
steps. This question can be effectively settled. Similarly, if $x$
is the power of a prime of odd index, $x=p_{2n+1}^{k}$, then $x$
belongs to $\mathcal{A}$ if and only if the $n$-th Turing machine
stops in exactly $k$ calculations steps, this also can be determined. 

Of course, $\mathcal{A}\,mod\,m$ does not depend recursively of $m$.
Indeed, the question: ``does $0$ belong to $\mathcal{A}\,mod\,p_{2n+1}$?''
is, by construction, equivalent to ``does the $n$-th Turing machine
halt?''.

Finally, we show that $\mathcal{A}$ is a closed set, which is equivalent
to finding, for any $x$ not in $\mathcal{A}$, a number $y$ such
that $x+y\mathbb{Z}$ does not meet $\mathcal{A}$. If $x$ has several
prime divisors, then $x+x\mathbb{Z}$ works, because any element of
it has several prime divisors. If $x$ is the power of a prime of
even index, $x=p_{2n}^{k}$, and $x$ is not in $\mathcal{A}$, it
must be that the $n$-th Turing machine stops in strictly less than
$k$ steps. Thus the only elements in $\mathcal{A}$ that are multiples
of $p_{2n}$ will have a valuation in $p_{2n}$ lower than $k$. Thus
$x+x\mathbb{Z}$ will also work. The last case is if $x$ is the power
of a prime of odd index, $x=p_{2n+1}^{k}$. In this case, we claim
that $x+p_{2n+1}x\mathbb{Z}$ does not meet $\mathcal{A}$. Indeed,
$x$ is the only power of $p_{2n+1}$ contained in $x+p_{2n+1}x\mathbb{Z}$,
all other elements of it have at least two different prime divisors. 
\end{proof}
With this we end the proof of Theorem \ref{thm:Thm1}: 
\begin{proof}
[Proof of Theorem \ref{thm:Thm1}]By Proposition \ref{prop:Main-prop-1-2-3},
for a subset $\mathcal{A}$ of $\mathbb{Z}$, Dyson's group $L(\mathcal{A})$
satisfies the requirements of Theorem \ref{thm:Thm1} provided that
$\mathcal{A}$ is closed, recursive, and that the function $\mathcal{A}\,mod\,n$
is not computable. Lemma \ref{lemma for thm 1} provides such a set. 
\end{proof}
We finally remark that an upper bound to the depth function of the
obtained group $L(\mathcal{A})$ can be effectively computed, and
that the group obtained in Theorem \ref{thm:Thm1} can be supposed
to have recursive depth function. 
\begin{cor}
\label{prop:Last Prop}There exists a group with solvable word problem
and recursive depth function, that does not have CFQ. 
\end{cor}

\begin{proof}
It appears clearly in the proof of Lemma \ref{lemma for thm 1} that
the constructed $\mathcal{A}$ is \emph{effectively closed}: if $x$
does not belong to it, then some $y$ such that $x+y\mathbb{Z}$ does
not meet $\mathcal{A}$ can effectively be found. Going back to the
proof of the first point of Proposition \ref{prop:Main-prop-1-2-3},
it appears that, given a non-identity element $w$ of $L(\mathcal{A})$,
the recursiveness of $\mathcal{A}$ permits to compute its normal
form. Then, in the proof of the second point of that same proposition,
it appears that, from this normal form and the effective closeness
of $\mathcal{A}$, some integer $N$ can be effectively found, such
that $w$ is non trivial in $L(\mathcal{A})_{N}$. A presentation
of $L(\mathcal{A})_{N}$ cannot necessarily be found, but there are
$2^{N}$ possible finite presentations for it, all of them with recursive
depth function. Taking the supremum of those depth functions allows
to compute a recursive upper bound to the depth function of $L(\mathcal{A})$. 

Although it is not clear whether that depth function is recursive
or not, by taking the direct product of the group $L(\mathcal{A})$
with a finitely presented group with recursive depth function greater
that that of $L(\mathcal{A})$, which exists by \cite{Kharlampovich2017},
one obtains a group which still has solvable word problem and uncomputable
finite quotients, and whose depth function is recursive. 
\end{proof}
In \cite{Rauzy2020}, we construct, also using Dyson's groups, a residually
finite group $G$ with solvable word problem, that not only does not
have CFQ, but that also is not \emph{effectively residually finite}:
there can be no algorithm that, given a non-trivial element $w$,
gives a finite quotient $(F,f)$ in which the image of $w$ is non-trivial.
This is done by constructing a closed subset $\mathcal{A}$ of $\mathbb{Z}$
that is not effectively closed. 

\bibliographystyle{abbrv}
\bibliography{CFQrefs}

\end{document}